\documentclass[a4paper,10pt,reqno, english]{amsart}

\usepackage{amsmath,amssymb,amscd,amsthm,amsfonts}
\usepackage{graphicx,subfigure}
\usepackage{hyperref}
\usepackage{dsfont}
\usepackage[nobysame, alphabetic]{amsrefs}
\usepackage{tikz}
\usepackage[capitalise]{cleveref}
\usepackage{mathrsfs}
\usepackage{enumitem}

\newtheorem{theorem}{Theorem}
\newtheorem{lemma}{Lemma}

\newtheorem{corollary}{Corollary}

\newtheorem{definition}{Definition}

\newcommand{\ff}{\mathcal{F}}

\def\rr{\mathds{R}}

\DeclareMathOperator{\lin}{lin}

\DeclareMathOperator{\conv}{conv}
\DeclareMathOperator*{\bigast}{\mathop{\scalebox{1.5}{$\ast$}}}

\newcommand{\kk}{\mathcal{K}}
\newcommand{\LL}{\mathscr{L}}

\newcommand{\dist}{\operatorname{dist}}

\title{Helly and Radon theorems for convex intersections containing $k$-flats}

\hypersetup{
  pdftitle={},
  pdfauthor={}
}

\author[Ludwigson]{Sarah Ludwigson}\address{Connecticut College, 270 Mohegan Ave, New London, CT 06320}
\email{sludwigso@conncoll.edu}

\author[Sober\'on]{Pablo Sober\'on}\address{Baruch College \& The Graduate Center, City University of New York, One Bernard Baruch Way, New York, NY 10010, United States} 
\email{psoberon@gc.cuny.edu}

\thanks{The research of S. Ludwigson was supported by NSF grant DMS-2349366 and by Jane Street.  The research of P. Sober\'on is supported by NSF CAREER grant DMS-2237324 and a PSC-CUNY Track 1 award.}



\begin{document}

\begin{abstract}
We study versions of results in combinatorial geometry related to families of convex sets in $\mathbb{R}^d$ whose intersection contains a $k$-dimensional affine space. We prove generalizations of the colorful Radon theorem, the fractional Helly theorem, the colorful Helly theorem, and the selection-structure Helly theorem.  When $k=0$, our arguments give new proofs of the corresponding versions for points.
\end{abstract}

\maketitle

\section{Introduction}

The intersection patterns of convex sets in $\rr^d$ is a central theme in discrete geometry.  There are two major families of results that describe how convex sets intersect.  One family, Helly-type results, describes how intersection properties of families of convex sets are consequences of conditions on subfamilies of bounded size \cites{Amenta2017, Holmsen:2017uf, Barany2022}.  The second family, Tverberg-type results, describes the intersection properties of families that are generated from sets of points (for example, the convex hulls of all subsets) \cites{Blagojevic2017, Barany2018, Barany2021}.

Recently, there has been renewed attention to quantitative results regarding the intersection properties of convex sets.  In other words, rather than proving that the intersection of some families of convex sets is not empty, we show that it is quantifiably large.  This can mean having large volume, having large diameter, containing many integer points, etc (see, e.g., \cites{Naszodi2016, Damasdi2021, Sarkar2021} and the references therein).

The goal of this manuscript is to prove new results regarding families of convex sets in $\rr^d$ whose intersection contains a $k$-dimensional affine space, which we call a $k$-flat. There is a version of Helly's theorem for this property, due to De Santis.  There are Helly-type results for containing cones, rays, or half-flats, but those have very different Helly numbers \cites{Kolodziejczyk1984, Barany2024, Ivanov2026}.

\begin{theorem}[De Santis 1957 \cite{DeSantis1957}]\label{thm:desantis-helly}
    Let $0 \le k \le d$ be integers and let $\ff$ be a finite family of closed convex sets in $\rr^d$.  If the intersection of every $d-k+1$ or fewer sets in $\ff$ contains a $k$-flat, then the intersection of $\ff$ contains a $k$-flat.
\end{theorem}

The case $k=0$ is Helly's theorem \cite{Helly:1923wr}.  The proof by De Santis is based on the following Radon-type theorem for $k$-flats.

\begin{theorem}[De Santis 1957 \cite{DeSantis1957}]\label{thm:desantis-radon}
    Let $k \le d$ be non-negative integers.  Given $d-k+2$ affine flats of dimension $k$ in $\rr^d$, there exists a partition of them into two sets $A,B$ such that $\conv(\cup A) \cap \conv (\cup B)$ contains a $k$-flat.
\end{theorem}

Again, the case $k=0$ is Radon's theorem \cite{Radon:1921vh}, which laid the foundation for the study of Tverberg-type results.  Another way it reduces to Radon's theorem is if all the $d-k+2$ flats are parallel.  In that case, the result is equivalent to Radon's theorem in the common orthogonal complement, of dimension $d-k$.  Our first main result is a version of colorful Radon for $k$-flats.

\begin{theorem}\label{thm:kflat-colorful-Radon}
    Let $0\le k\le d$ be integers.  Given $d-k+1$ pairs of $k$-flats in $\rr^d$, there exists a partition of them into two sets $A,B$ such that each of $A,B$ has one $k$-dimensional flat from each pair, and such that $\conv (\cup A) \cap \conv (\cup B)$ contains a $k$-flat.
\end{theorem}

The result is called colorful since we can think of each pair as a color class.  The parts of the partition have exactly one $k$-flat of each color.  The case $k=0$ is known as the colorful Radon theorem.  This result was proved by Lov\'asz using the Borsuk--Ulam theorem \cite{Barany:1992tx}, and a second proof using elementary linear algebra was given by the second author \cite{Soberon:2015cl}.  Our proof of \cref{thm:kflat-colorful-Radon}, applied when $k=0$, is a new proof of the colorful Radon theorem.  In particular, it shows that colorful Radon is a direct consequence of Radon's theorem.

Additionally, we prove several generalizations of \cref{thm:desantis-helly}.  We prove fractional and colorful versions of the De Santis theorem, which generalize the colorful and fractional versions of Helly's theorem.

\begin{theorem}[Colorful De Santis]\label{thm:colorful-desantis}
    Let $0 \le k \le d$ be integers, and let $\ff_1,\dots, \ff_{d-k+1}$ be finite families of closed convex sets in $\rr^d$.  If the intersection of every $(d-k+1)$-tuple $F_1 \in \ff_1,\dots, F_{d-k+1}\in \ff_{d-k+1}$ contains a $k$-flat, then there is some family $\ff_i$ such that $\bigcap \ff_i$ contains a $k$-flat.
\end{theorem}

This is also known as a colorful theorem, as we can consider each $\ff_i$ as a color class.  The condition says that if the intersection of every rainbow family contains a $k$-flat, there must be a color whose intersection contains a $k$-flat.  The case $\ff_1 = \dots = \ff_{d-k+1}$ of \cref{thm:colorful-desantis} is precisely \cref{thm:desantis-helly}.  Our proof applied to this case gives a new proof of \cref{thm:desantis-helly} that does not use the Radon theorem for $k$-flats.

\begin{theorem}[Fractional De Santis]\label{thm:frational-desantis}
    Let $0 \le k \le d$ be integers and $0< \alpha \le 1$ be a real number.  There exists a constant $\beta = \beta(\alpha,k,d) > 0$ such that the following holds.  For every finite family $\ff$ of at least $d-k+1$ closed convex sets in $\rr^d$, if the intersection of at least $\alpha \binom{|\ff|}{d-k+1}$ of the $(d-k+1)$-tuples of $\ff$ contain a $k$-flat, then there exists $\mathcal{G} \subseteq \ff$ such that $|G|\ge \beta |\ff|$ and $\bigcap \mathcal{G}$ contains a $k$-flat.  Moreover $\beta \ge \frac{\alpha}{d-k+1}.$
\end{theorem}

The case $k=0$ of \cref{thm:colorful-desantis} is the colorful Helly theorem by Lov\'asz \cite{Barany1982}, and the case $k=0$ of \cref{thm:frational-desantis} is the fractional Helly theorem of Katchalski and Liu \cite{Katchalski1979}.  The standard proof of both results relies on the idea of choosing a direction $v$ and finding the $v$-directional minimum of a convex set.  Then, a small family whose intersection has the largest $v$-directional minimum implies the conclusion required by the theorem (either a full family $\ff_i$ containing that point in the case of colorful Helly, or a large subfamily containing that point in the case of fractional Helly).  For our purposes, we cannot use such a minimization, since $k$-flats may not be bounded in the directions we need.

To fix this issue, we introduce an ordering on the space of affine flats of $\rr^d$.  Our main technical lemma, \cref{lem:minimize-small-family}, says that the minimum flat (in this order) contained in an intersection of convex sets is determined by a small subfamily of convex sets.  This allows us to prove both theorems listed above.  When we set $k=0$, our methods give new proofs of the colorful Helly theorem and fractional Helly theorem, since the minimization process is different.

Finally, there are now broad generalizations of the colorful Helly theorem.  A well-known example is the matroidal Helly theorem by Kalai and Meshulam \cite{Kalai2005}, which replaces the coloring conditions from the colorful Helly theorem by conditions related to the independence complex of a matroid.  The second author recently generalized the Kalai--Meshulam Helly theorem to selection structures \cite{soberon2026kkm-arxiv}.  These are structures that generalize the independence complexes of matroids.  They allow us to prove Helly-type results where the coloring condition are encoded by chessboard complexes or independence complexes.  

We prove a selection-structure generalization of De Santis' theorem in \cref{thm:selection-structure-desantis}.  We defer the statement of this result until \cref{sec:selection-structures}, as it requires additional definitions.  \cref{thm:selection-structure-desantis} generalizes our colorful Helly theorem, \cref{thm:colorful-desantis}.  We keep a proof of \cref{thm:colorful-desantis} in \cref{sec:proofs-for-hellys} as it is easier to read and requires fewer definitions.  If we set $k=0$, \cref{thm:selection-structure-desantis} generalizes the selection-structure Helly theorem of the second author \cite{soberon2026kkm-arxiv}*{Thm 6}, which in turn implies the Kalai--Meshulam topological colorful Helly theorem applied to families of convex sets in $\rr^d$.

We first prove the colorful Radon theorem for $k$-flats in \cref{sec:colorful-radon}.  We introduce our new minimization argument in \cref{sec:order-of-k-flats}, and use it to prove \cref{thm:colorful-desantis} and \cref{thm:frational-desantis} in \cref{sec:proofs-for-hellys}.  We define selection structures and prove \cref{thm:selection-structure-desantis} in \cref{sec:selection-structures}.  We conclude with remarks in \cref{sec:remarks}.

\section{Colorful Radon for affine flats}\label{sec:colorful-radon}

\begin{proof}[Proof of \cref{thm:kflat-colorful-Radon}]
    Let $v_1,\dots, v_{d-k+1}$ be the vertices of a regular simplex in $\rr^{d-k}$.  Let $F_1,\dots, F_{d-k+1}$ be the pairs of affine spaces given to us.

    For each affine flat $H \in F_i$, let $\tilde{H} \subset \rr^{d}\times \rr^{d-k} = \rr^{2d-k}$ be the set of points obtained from appending the coordinates of $v_i$ to each point of $H$.  Note that $\tilde{H} \subseteq \rr^{2d-k}$ is also a $k$-dimensional affine flat.  See \cref{fig:append} for an illustration.

    \begin{figure}
        \centering
        \includegraphics[]{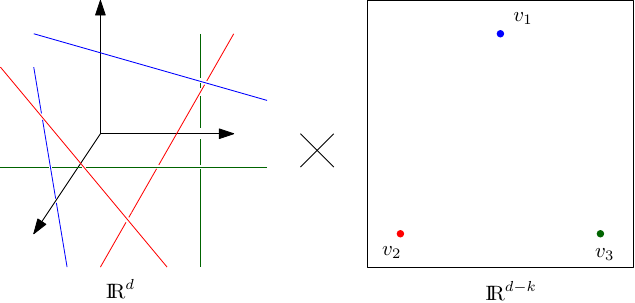}
        \caption{An example of the construction with $d=3,k=1$.  We have three pairs of lines in $\rr^3$, and we consider the vertices $v_1,v_2,v_3$ of a triangle in the plane.  A line $\ell$ in the pair $F_i$ is replaced by the set $\ell \times \{v_i\}$.}
        \label{fig:append}
    \end{figure}

    Let $\pi_1: \rr^{2d-k}\to \rr^d$ be the projection to the first $d$ coordinates, and $\pi_2: \rr^{2d-k}\to \rr^{d-k}$ be the projection to the last $d-k$ coordinates.

    If we lift each flat in the hypothesis to $\rr^{2d-k}$, we obtain $2(d-k+1) = (2d-k)-k+2$ affine spaces of dimension $k$ in $\rr^{2d-k}$.  We can apply \cref{thm:desantis-radon} to obtain a partition of them into two sets $A, B$ such that $\conv(\cup A) \cap \conv (\cup B)$ contains a $k$-dimensional flat.

    Let us verify that this partition has exactly one flat from each set.  For each flat $\tilde{H}$, its projection $\pi_2(\tilde{H})$ is a single point in $\rr^{d-k}$.  Notice that $\pi_2(A) \cup \pi_2(B)$ consists of two copies of $v_1,\dots, v_{d+1-k}$.  A point $p \in \conv \{v_1,\dots,v_{d-k+1}\}$ can be written as a convex combination of $v_1,\dots,v_{d-k+1}$ in a unique way.  Therefore, if $x \in \conv(\cup A) \cap \conv(\cup B)$, and $\pi_2(x)$ is written as a convex combination of $v_1,\dots, v_{d-k+1}$, for any $v_j$ that receives a positive coefficient in the convex combination from $\pi_2(A)$ that gives $x$, we must have $v_j \in \pi_2(B)$ with the same coefficient.  In other words, every time we use an element from $F_j$ in $A$, we need an element from $F_j$ in $B$, and vice versa.

    Let $L$ be a $k$-dimensional flat in $\conv(\cup A) \cap \conv(\cup B)$.  We know $\pi_2(L) \subset \conv (v_1,\dots, v_{d-k+1})$.  Therefore, $\pi_2(L)$ is bounded.  This means that $\pi_1(L)$ is a $k$-dimensional flat.  If we look at the partition induced on $\pi_1(A) \cup \pi_1(B)$, it satisfies the conditions we were looking for.  
\end{proof}

\section{Ordering flats in $\rr^d$}\label{sec:order-of-k-flats}

Let $L$ be an affine flat in $\rr^d$.  Let $r(L) = (d-\dim L, \dist(\bar{0},L))$ be the pair associated to $L$.  This allows us to order affine flats in $\rr^d$, by ordering their associated pairs lexicographically.  In other words, 
\[
(x,y) \le (x',y') \leftrightarrow \begin{cases}
    x < x' & \quad \mbox{or}  \\
    x=x' \mbox{ and } y \le y'.
\end{cases}
\]
Let $K \subseteq \rr^d$ be a closed convex set.  Let $M(K)$ be the affine flat $L\subseteq K$ such that $r(L)$ is minimal.  First we restrict ourselves to flats of the highest possible dimension, and then we choose the one closest to the origin.

\begin{lemma}
    For every closed convex set $K \subseteq \rr^d$, the affine flat $M(K)$ exists and is unique.
\end{lemma}

\begin{proof}
    Assume for a contradiction that $K$ has two minimal affine spaces $L, L'$.  Then, they must have the same dimension.  Since $K$ is closed, we also have that $L$ and $L'$ must be parallel.  Otherwise, $K$ would contain a translation of $L \oplus L'$, where $\oplus$ denotes the Minkowski sum, and this would be a higher-dimensional affine space contained in $K$, contradicting the minimality of $r(L)$ and $r(L')$.  Let $V = L^{\perp}$, and $\pi : \rr^d \to V$ be the orthogonal projection.  Since $L$ is the affine space of the largest possible dimension contained in $K$, we know that $\pi(K)$ contains no lines and is convex.  Moreover, $\pi(L)$ and $\pi(L')$ are points.

    Additionally, $\dist(\bar{0},L) = \dist (\bar{0}, \pi(L)) = \|\pi (L)\|$, and the same holds for $\pi(L')$.  However, since $\pi(K)$ is closed and convex, there is a unique point $p \in \pi(K)$ minimizing the distance to the origin (which could be the origin itself).  The flat $L''$ parallel to $L$ through $p$ satisfies $r(L'') \le r(L)$ and $r(L'') \le r(L')$, and equality can only happen if $L=L'=L''$, contradicting the assumption that $K$ had two different minimal affine flats.
\end{proof}

\begin{lemma}\label{lem:minimize-small-family}
Let $\ff$ be a finite family of closed convex sets in $\rr^d$ with non-empty intersection, and let $L = M(\bigcap \ff)$, and $\ell = \dim L$.  Then, there exists a family $\mathcal{G} \subseteq \ff$ of at most $d-\ell$ sets such that $L = M(\bigcap \mathcal{G})$.    
\end{lemma}

Notice that if $\bigcap \ff$ contains a $k$-dimensional flat, we have $\ell \ge k$, which implies $|\mathcal{G}| \le d-\ell \le d-k$.

For a convex set $K \subseteq \rr^d$, let $\lin(K)$ be the largest linear space contained in some translate of $K$.  The space $\lin(K)$ is independent of the translation used, as long as it contains $\bar{0}$.

\begin{proof}
    We split the proof into two cases depending on whether $\bar{0} \in \bigcap \ff$ or $\bar{0} \not\in \bigcap \ff$.  Let $V = L^{\perp}$, let $\pi:\rr^d \to V$ be the orthogonal projection, and $K = \pi (\bigcap \ff)$.

    \textbf{Case 1.} $\bar{0} \in \bigcap \ff$

Let $s(K)$ be the set of non-zero vectors $v \in V$ such that $\pi(F)$ is contained in a half-space orthogonal to 
$v$ for some $F \in \ff$.

We claim that $\operatorname{span}(s(K)) = V$.  Otherwise, there would be a non-zero vector $u \in V$ that is orthogonal to all $v \in s(K)$.  Since $\pi(K)$ contains no lines, there is some scalar $\alpha$ (which could be positive or negative), such that $\alpha u \not\in K = \bigcap \ff$.  Therefore, there exists $F \in \ff$ such that $\alpha u \not\in \pi(F)$.  There must be a hyperplane separating $\pi(F)$ from $\alpha u$.  The normal vector to this hyperplane must be an element of $s(K)$.  Since this hyperplane would also be separating $0$ from $\alpha u$, its normal vector cannot be orthogonal to $v$, contradicting the choice of $u$.

Since $\operatorname{span}(s(K))= V$, we can choose a basis of $V$ from those vectors.  Let $v_1,\dots,v_{d-\ell}$ be the basis we selected, assigned to the sets $F_1,\dots, F_{d-\ell}$ in $\ff$, respectively.  By the same arguments as above, the set $\bigcap_{i=1}^{d-\ell}\pi(F_i)\subseteq V$ contains no lines.  The family $\mathcal{G} = \{F_1,\dots,F_{d-\ell}\}$ is the family we were searching for.

 \textbf{Case 2.} $\bar{0} \not\in \bigcap \ff$

Let $p \in V$ be the closest point to the origin in $K$, and let $H_p =\{x \in V : \langle x, p \rangle < \langle p, p\rangle \}$.  We have $K \cap H_p = \emptyset$, as otherwise there would be points of $K$ closer to $\bar{0}$, see \cref{fig:construction-Hp}.  Consider $\ff' = \{\pi(F): F \in \ff\}\cup \{H_p\}$.  The family $\ff'$ has empty intersection.  By Helly's theorem, there must be a subfamily of size at most $d - \ell +1$ that has empty intersection.  Since $\ff$ is intersecting, the non-intersecting small set must consist of $H_p$ and $\{\pi (F): F \in \mathcal{G}'\}$ for some $\mathcal{G}' \subseteq \ff$ with $|\mathcal{G}'| \le d-\ell$.

We might still have $M(\mathcal{G}') \neq L$, as $M(\mathcal{G}')$ might be higher-dimensional.  Let $L' = M(\mathcal{G}')$ and $\ell' = \dim (L')$.  If $\ell' = \ell$, we are done.  Otherwise, we have $\ell'> \ell$.  By the previous case, replacing $0$ by $p$, we know that there exists $\mathcal{G}'' \subseteq \mathcal{G}'$ such that
$|\mathcal{G}''| \le d-\ell'$ 
and $M(\mathcal{G}'') = M (\mathcal{G}') = L'$.

Let $V' = \{x \in L': (x-p) \perp L\}$.  The space $V'$ has dimension $\ell'-\ell$.  The set $\pi'(\bigcap F)$ contains no lines.  An analogous argument as in the previous case shows that we can find a set $\mathcal{J}\subseteq \ff$ of at most $\ell'-\ell$ sets such that $\bigcap \{\pi'(F) : F \in \mathcal{J}\}$ contains no lines.  Finally, we can take $\mathcal{G} = \mathcal{G}'' \cup \mathcal{J}$, which is a family of at most $(d-\ell') + (\ell'-\ell) = d-\ell$ sets that satisfies the conditions we wanted.

\begin{figure}
    \centering
    \includegraphics[]{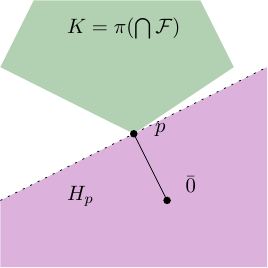}
    \caption{The construction of $H_p$.}
    \label{fig:construction-Hp}
\end{figure}

%    In this case, $L = \lin(\bigcap \ff)$.  We know that every set in $\mathcal{F}$ contains $L$ and we just need to find a family $\mathcal{G}$ that does not contain any other direction.  We will construct $\mathcal{G}$ inductively.  First, set $\mathcal{G} = \emptyset$ and $L' = L$.  We know that if  
    
 %   Let $v$ be a vector not contained in $L'$.
\end{proof}

\section{Proofs of the colorful and fractional Helly theorems}\label{sec:proofs-for-hellys}

In this section we prove \cref{thm:colorful-desantis} and \cref{thm:frational-desantis}.

\begin{proof}[Proof of \cref{thm:colorful-desantis}]
    Consider $\mathcal{A}$ the set of $(d-k)$-tuples of sets $F_1,\dots, F_{d-k}$ from $\ff_1 \cup \ff_2 \cup \dots \cup \ff_{d-k+1}$ that use at most one set from each $\ff_i$.  In other words, $\mathcal{A}$ is the set of all rainbow $(d-k)$-tuples.

    By the conditions of the theorem, the intersection of every $(d-k)$-tuple in $\mathcal{A}$  contains a $k$-flat.  Let $(F_1,\dots, F_{d-k}) \in \mathcal{A}$ be the $(d-k)$-tuple such that $L=M(\bigcap_{i=1}^{d-k}F_i)$ is maximal in the order introduced in \cref{sec:order-of-k-flats}.  We may assume without loss of generality that $F_1 \in \ff_1, \dots, F_{d-k}\in \ff_{d-k}$.  We know that $L$ has dimension greater than or equal to $k$.  We will show that every set $F_{d-k+1}\in \ff_{d-k+1}$ satisfies $L \subseteq F_{d-k+1}$.

    Suppose for a contradiction that there exists $F_{d-k+1}\in \ff_{d-k+1}$ such that $L \not\subseteq F_{d-k+1}$.  Let $B = \bigcap_{i=1}^{d-k+1}F_i$.

    First, since $B \subseteq \bigcap_{i=1}^{d-k}F_i$, we have $r(M(B)) \ge r(L)$.  However, since $L \not\subseteq B$, we have $M(B) \neq L$.  Therefore, $r(M(B)) > r(L)$.  By the conditions of the theorem, the dimension of $M(B)$ is at least $k$.  Therefore, by \cref{lem:minimize-small-family}, there is a subfamily of at most $d-k$ elements of $\{F_1,\dots, F_{d-k}\}$ such that the minimum flat in their intersection is equal to $M(B)$.  If this family has fewer than $d-k$ sets, we can add sets until we have $d-k$.  This will give us a family in $\mathcal{A}$ whose minimum is larger than $L$, contradicting the choice of $L$.  This means that every set in $\ff_{d-k+1}$ contains $L$, as we wanted to show.
\end{proof}

\begin{proof}[Proof of \cref{thm:frational-desantis}]
    Let $n = |\ff|$ and let $\mathcal{A} \subseteq \binom{\ff}{d-k+1}$ be the family of subsets of $d-k+1$ sets of $\ff$ whose intersection contains a $k$-flat.  Let $\mathcal{B} = \binom{\ff}{d-k}$.

    \cref{lem:minimize-small-family} induces a function $t: \mathcal{A} \to \mathcal{B}$ so that for $\ff' \in \mathcal{A}$, we have $t(\ff')$ is a subset $\mathcal{G}' \subset \ff'$ of $d-k$ sets such that $M(\bigcap \ff') = M(\bigcap \mathcal{G}')$.  By the pigeonhole principle, there exists $\mathcal{G}' \in \mathcal{B}$ whose preimage satisfies
    \[
    |t^{-1}(\mathcal{G}')| \ge \frac{|\mathcal{A}|}{|\mathcal{B}|} \ge \frac{\alpha \binom{n}{d-k+1}}{\binom{n}{d-k}} = \frac{\alpha(n-d+k)}{d-k+1} \ge \left(\frac{\alpha}{d-k+1}\right) n - (d-k).
    \]
    Let $L = M(\bigcap \mathcal{G}')$.  This is a flat of dimension greater than or equal to $k$.  Each of the $d-k$ sets of $\mathcal{G}'$ contains $L$.  Additionally, every $(d-k+1)$-tuple in $t^{-1}(\mathcal{G}')$ corresponds to an additional set in $\ff\setminus \mathcal{G}'$ that contains $L$.  Therefore, there are at least $\alpha\cdot  n/(d-k+1)$ sets in $\ff$ that contain $L$.  In other words, we can take $\beta = \alpha/(d-k+1)$.
\end{proof}

\section{Selection structures and $k$-flats}\label{sec:selection-structures}

As mentioned in the introduction, selection structures provide us with a way to generalize colorful Helly results beyond matroids.  Selection structures also allow us to generalize colorful versions of the Borsuk--Ulam theorem \cite{soberon2026borsuk-arxiv}.  The goal of this section is to prove the selection-structure version of Helly for $k$-flats.  Let us start with the key definitions.

\begin{definition}
    Let $W$ be a finite set of vertices and $\kk$ be a simplicial complex with vertices in $W$.  For $U \subseteq W$, let $\kk[U] = \{\sigma \cap U: \sigma \in \kk\}$ be the restriction of $\kk$ to $U$.  Additionally, if $D=\{D_w : w \in W\}$ is an indexed family of non-empty finite sets, we define
    \[
    \kk(U;D) = \bigcup_{\sigma \in \kk[U]}\left(\bigast_{w \in \sigma} D_w\right)
    \]
\end{definition}

Here, $\bigast$ represents an iterated join.  Intuitively, $\kk(U;D)$ is the complex $\kk[U]$, but each vertex is replaced by the corresponding set $D_w$.  With this, we are ready to define selection structures.  We call the sets $D_w$ the \textit{fibers}, and the space $\kk(U;D)$ a \textit{parallel expansion} of $\kk[U]$.

\begin{definition}
    Let $W$ be a finite set and $d$ be a non-negative integer.  A $d$-admissibile selection structure $\Sigma = (\kk, \LL)$ on $W$ is a pair in which $\kk$ is a simplicial complex on $W$ and $\LL$ is a non-empty upward-closed family of sets in $W$ with the following two properties:
    \begin{itemize}
        \item For every indexed family $D=\{D_w:w \in W\}$ of non-empty finite sets and every $U \in \LL$, the complex $\kk(U;D)$ is $(d-1)$-connected.
        \item For every $U \in \LL$, every facet of $\kk[U]$ is a facet of $\kk$.
    \end{itemize}
\end{definition}

With this definition, we can now state the selection-structure De Santis theorem.

\begin{theorem}[Selection-structure De Santis]\label{thm:selection-structure-desantis}
Let $W$ be a finite set, $0 \le k \le d$ be integers, and $\Sigma = (\kk, \LL)$ be a $(d-k)$-admissible selection structure on $W$.  Suppose that for each $w \in W$ we have a closed convex set $F_w \subseteq \rr^d$.  Additionally, assume that for every $\sigma \in \kk$ we have that $\bigcap_{w \in \sigma}F_w$ contains a $k$-flat.  Then, there exists $W' \subseteq W$ such that $W \setminus W' \not\in \LL$ and $\bigcap_{w \in W'}F_w$ contains a $k$-flat.  
\end{theorem}

As an example, consider a set $W = W_1 \sqcup W_2 \sqcup \dots \sqcup W_{d-k+1}$ partitioned into $d-k+1$ sets.  Take $\kk$ to be the family of sets that have at most one element from each $W_i$ and $\LL$ to be the family of sets that have at least one element from each $W_i$.  Then $\Sigma = (\kk, \LL)$ is a $(d-k)$-admissible selection structure.  Moreover, this precise selection structure gives us exactly the setting for \cref{thm:colorful-desantis}.

As an additional example, consider $M$ the independence complex of a matroid with vertex set $W$, with rank greater than or equal to $d-k+1$.  We can consider $\kk$ the $(d-k)$-skeleton of $M$, and $\LL$ be the family of sets of rank greater than or equal to $d-k+1$.  This is a $(d-k)$-admissible selection structure (see \cite{soberon2026kkm-arxiv}).  This gives us the following matroid version, which generalizes the Kalai--Meshulam colorful Helly theorem (applied to families of convex sets).

\begin{corollary}
    Let $W$ be a finite set and $0 \le k \le d$ be integers.  Let $\mathcal{M}$ be a matroid on $W$ with rank function $\rho$.  Assume that for each $w \in W$ we have a closed convex set $F_w \subseteq \rr^d$.  Additionally, assume that for every $\sigma \subseteq W$ that is independent in $\mathcal{M}$ we have $\bigcap_{w \in \sigma}F_w$ contains a $k$-flat.  Then, there exists $W' \subseteq W$ such that $\rho (W \setminus W') \le d-k$ and $\bigcap_{w \in W'}F_w$ contains a $k$-flat.
\end{corollary}

To prove \cref{thm:selection-structure-desantis}, we also use the minimizing argument from \cref{sec:order-of-k-flats}.

\begin{proof}[Proof of \cref{thm:selection-structure-desantis}]
    To simplify the notation, for every $\sigma \in \kk$, let $M(\sigma) = M(\bigcap_{w \in \sigma}F_w)$.  By the condition of the theorem, for every $\sigma \in \kk$, the dimension of $M(\sigma)$ is at least $k$.

    Let $B$ be the facet of $\kk$ for which $L=M(B)$ is maximal in the order introduced in \cref{sec:order-of-k-flats}.  By \cref{lem:minimize-small-family}, there exists $\sigma \subseteq B$ such that $|\sigma| \le d - \dim (L) \le d-k$ and $M(\sigma) = M(B)$.

    Now let $\tau = \{w \in W: L \subseteq F_w\}$.  If $W \setminus \tau \not\in \LL$, we can take $W' = \tau$ and we are done.  Assume for a contradiction that $W \setminus \tau \in \LL$.  Since $\LL$ is upward-closed, we have $U = W \setminus \tau \cup \sigma \in \LL$.  Note that in $U$, the elements of $\sigma$ are the only sets that correspond to convex sets that contain $L$.

    Moreover, no face of $\kk[U]$ properly contains $\sigma$.  If there was such a face $\sigma'$, by monotonicity we would have $r(M(\sigma')) \ge r(M(\sigma))=r(L)$.  Since only the sets corresponding to $\sigma$ contain $L$, we would have $r(M(\sigma')) > r(M(\sigma))$, which would contradict the choice of $L$.

    Now consider the family of sets $D=\{D_w : w \in W\}$ defined by
    \[
    D_w = \begin{cases}
        \{0,1\} & \mbox{if }w \in \sigma, \\
        \{0\} & \mbox{otherwise.}
    \end{cases}
    \]
    Now consider the space $\kk(U;D)$.  The elements of $\sigma$ correspond to a $(|\sigma|-1)$-sphere in the parallel expansion (to put it precisely, $\kk(\sigma;D)$ is a $(|\sigma|-1)$-dimensional sphere).  Since there is no face that properly contains $\sigma$ in $U$, we have that the $(|\sigma|-1)$-st reduced homology of $\kk(U;D)$ is non-trivial, which contradicts the connectedness condition.  Therefore, the conclusion of the theorem must hold.
\end{proof}

\section{Remarks}\label{sec:remarks}

The optimal value for $\beta$ in the fractional Helly theorem is known \cites{Eckhoff1985, Kalai1986}.  However, it is not clear whether the natural version of those bounds (replacing every instance of $d$ by $d-k$) would apply for \cref{thm:frational-desantis}.

To prove those bounds, it would be sufficient to show the following.  Let $\ff$ be a finite family of convex sets in $\rr^d$, each containing a flat of dimension at least $k$.  We can form a simplicial complex $N_k(\ff)$ as follows.  Each set in $\ff$ corresponds to a vertex of $N_k(\ff)$, and a set of vertices makes a face if and only if the intersection of its corresponding sets contains a $k$-flat.  The complex $N_0(\ff)$ is the nerve complex of the family, and is known to be $d$-collapsible \cite{Wegner1975}.  

If we could show that $N_k(\ff)$ is $(d-k)$-collapsible, then we would be able to apply directly the results of Kalai to obtain much better bounds for $\beta$ in \cref{thm:frational-desantis}, and \cref{thm:selection-structure-desantis} would be a direct consequence of \cite{soberon2026kkm-arxiv}*{Thm 12}.

The authors do not know if $N_k(\ff)$ is $(d-k)$-collapsible, and the standard swiping arguments used to show the collapsibility of $N_0(\ff)$ seem to fail.  We consider it interesting that it is still possible to prove the fractional and selection-structure results despite the lack of bounds on the collapsibility of this simplicial complex. 


\end{document}